\documentclass[11pt, twoside, leqno]{article}

\usepackage{amsmath}
\usepackage{amssymb}
\usepackage{titletoc}
\usepackage{mathrsfs}
\usepackage{amsthm}
\usepackage{color}
\usepackage{latexsym,amssymb}

\usepackage{indentfirst}
\usepackage{color}
\usepackage{txfonts}

\allowdisplaybreaks

\newtheorem{theorem}{Theorem}[section]

\theoremstyle{definition}
\newtheorem{remark}[theorem]{Remark}

\renewcommand{\appendix}{\par
   \setcounter{section}{0}%
   \setcounter{subsection}{0}%
   \setcounter{subsubsection}{0}%
   \gdef\thesection{\@Alph\c@section}%
   \gdef\thesubsection{\@Alph\c@section.\@arabic\c@subsection}%
   \gdef\theHsection{\@Alph\c@section.}%
   \gdef\theHsubsection{\@Alph\c@section.\@arabic\c@subsection}%
   \csname appendixmore\endcsname
 }

\numberwithin{equation}{section}

\begin{document}

\arraycolsep=1pt

\title{\bf\Large
New function classes of Morrey-Campanato type and their applications
\footnotetext{\hspace{-0.35cm} 2020 {\it
Mathematics Subject Classification}. Primary 46E36;
Secondary 46E35, 42B25, 42B20, 42B35, 30L99.
\endgraf {\it Key words and phrases.} Morrey-Campanato space, BMO function, H\"{o}lder continue functions,
Commutator, maximal operators.
\endgraf This project is partially supported by
the National Natural Science Foundation of China(No. 11971237, 12071223, 12101010) and Youth Foundation of Anhui Province (No. 2108085QA19)}}
\date{}
\author{Dinghuai Wang, Lisheng Shu}
\maketitle

\vspace{-0.8cm}

\begin{center}
\begin{minipage}{14cm}
{\small {\bf Abstract}\quad
The aim of this paper is to introduce and investigative some new function classes of Morrey-Campanato type. Let $0<p<\infty$ and $0\leq \lambda<n+p$. We say that $f\in \mathcal{\bar{L}}^{p,\lambda}(\Omega)$ if
$$\sup_{x_{0}\in \Omega,\rho>0}\rho^{-\lambda}\int_{\Omega(x_{0},\rho)}\big|f(x)-|f|_{\Omega(x_{0},\rho)}\big|^pdx<\infty,$$
where $\Omega(x_{0},\rho)=Q(x_{0},\rho)\cap \Omega$ and $Q(x,\rho)$ is denote the cube of $\mathbb{R}^n$. Some basic properties and characterizations of these classes are presented. If $0\leq \lambda<n$, the space is equivalent to related Morrey space. If $\lambda=n$, then $f \in \mathcal{\bar{L}}^{p,n}(\Omega)$ if and only if $f\in BMO(\Omega)$ with $f^{-}\in L^{\infty}(\Omega)$, where $f^{-}=-\min\{0,f\}$. If $n<\lambda\leq n+p$, the $\mathcal{\bar{L}}^{p,\lambda}(\Omega)$ functions establish an integral characterization of the nonnegative H\"{o}lder continue functions. As applications, this paper gives unified criterions on the necessity of bounded commutators of maximal functions.}
\end{minipage}
\end{center}

\vspace{0.1cm}

\tableofcontents

\vspace{0.1cm}

\section{Introduction}

The domains $\Omega \subset \mathbb{R}^n$ in this paper are supposed to satisfy the following property: there exists a constant $A>0$ such that for all $x_{0}\in \Omega, \rho<{\rm diam} \ \Omega$ we have
\begin{equation}\label{A}
|Q(x_{0},\rho)\cap \Omega|\geq A\rho^n,
\end{equation}
where $Q(x,\rho)$ is denote the cube of $\mathbb{R}^n$, with sides parallel to the coordinate axes, having a center at $x$ and side $2\rho$. We set $\Omega(x_{0},\rho):=Q(x_{0},\rho)\cap \Omega$. Note that every domain of class $C^1$ or Lipschitz has the above property.

Let $0< p<+\infty$ and $\lambda\geq 0$. Define the Morrey-Campanato space
\begin{equation*}\label{A}
\mathcal{L}^{p,\lambda}(\Omega):=\Big\{f\in L^{p}(\Omega):\sup_{x_{0}\in \Omega,\rho>0}\rho^{-\lambda}\int_{\Omega(x_{0},\rho)}|f(x)-f_{\Omega(x_{0},\rho)}|^{p}dx<\infty\Big\},
\end{equation*}
endowed with the seminorm defined by
$$\|f\|^{p}_{\mathcal{L}^{p,\lambda}(\Omega)}:=\sup_{x_{0}\in \Omega,\rho>0}\rho^{-\lambda}\int_{\Omega(x_{0},\rho)}|f(x)-f_{\Omega(x_{0},r)}|^pdx.$$
%and the norm
%$$\|f\|_{\mathcal{L}^{p,\lambda}(\Omega)}:=\|f\|_{\mathcal{L}^{p,\lambda}(\Omega)}+\|f\|_{L^{p}(\Omega)}.$$

We list some known results on the structure of the Morrey-Campanato spaces.
% We omit the assumptions on $\Omega$ to be imposed in each case.
\begin{itemize}
\item [(i)]
$\lambda=0.$ It is obvious that $\mathcal{L}^{p,0}(\Omega)=L^{p}(\Omega)$.

\item [(ii)]$0<\lambda<n.$ $\mathcal{L}^{p,\lambda}(\Omega)$ is equivalent to the Morrey space $L^{p,\lambda}(\Omega)$, i.e.
$$L^{p,\lambda}(\Omega):=\{f\in L^{p}(\Omega):\sup_{x_{0}\in \Omega,\rho>0}\rho^{-\lambda}\int_{\Omega(x_{0},\rho)}|f(x)|^{p}dx<\infty\},$$
endowed with the norm defined by
$$\|f\|^p_{L^{p,\lambda}(\Omega)}:=\sup_{x_{0}\in \Omega,\rho>0}\rho^{-\lambda}\int_{\Omega(x_{0},\rho)}|f(x)|^{p}dx.$$
This was proved by Campanato \cite{C1964}.

\item [(iii)] $\lambda=n.$ $\mathcal{L}_{1,n}(\mathbb{R}^n)=BMO(\mathbb{R}^n)$, the spaces of bounded mean oscillation. The crucial property of $BMO$ functions is the John-Nirenberg inequality \cite{JN1961},
$$|\{x\in Q: |f(x)-f_{Q}|>\lambda\}|\leq c_{1}|Q|e^{-\frac{c_{2}\lambda}{\|f\|_{BMO(\mathbb{R}^n)}}}, \lambda>0,$$
where $c_{1}$ and $c_{2}$ depend only on the dimension. A well-known immediate corollary of the John-Nirenberg inequality as follows:
$$\|f\|_{BMO(\mathbb{R}^n)}\approx \sup_{Q}\frac{1}{|Q|}\Big(\int_{Q}|f(x)-f_{Q}|^{p}dx\Big)^{1/p},$$
for all $1<p<\infty$. In fact, the equivalence also holds for $0<p<1$. See, for example, the work of Str\"{o}mberg \cite{Str1979}(or \cite{HT2019} and \cite{WZTams} for the general case).

\item [(iv)] $n<\lambda\leq n+p$. $\mathcal{L}_{1,n}(\Omega)=C^{0,\alpha}(\Omega)$ with $\alpha=(\lambda-n)/p$.
For $0<\alpha\leq 1$, the H\"{o}lder continuous functions $C^{0,\alpha}(\Omega)$ is the set of functions $f$ such that
$$\|f\|_{C^{0,\alpha}(\Omega)}:=\sup_{x,y\in \Omega \atop_{x\neq y}}\frac{|f(x)-f(y)|}{|x-y|^{\alpha}}<\infty.$$
This was shown independently by Campanato \cite{C1963} and by Meyers \cite{M1964} for $1\leq p<\infty$, and by the first author, Zhou and Teng \cite{WZTmn} for $0<p<1$.
\end{itemize}

The Morrey-Campanato spaces on Euclidean spaces play an important role in the study of partial differential equation; see \cite{M1967} and \cite{P1969}. Campanato spaces are useful tools in the regularity theory of PDEs as a result of their better structures, which allow us to give an integral characterization of the spaces of H\"{o}lder continuous functions. This also allows generalization of the classical Sobolev embedding theorems \cite{L2007, L1995, L1998}. It also well known result that Campanato space is the dual space of related Hardy space \cite{T1992}. Other types of Morrey-Campanato function have also received attention. X. T. Duong and L. X. Yan \cite{DY2005} introduced new function spaces of $BMO$ type, which are defined by variants of maximal functions associated with generalized approximations to the identity. Later, Tang \cite{T2007} generalized the results to the new function spaces of Morrey-Campanato type.
In this paper, we introduce some new function classes of Morrey-Campanato type as follows. Let $0< p<+\infty$ and $\lambda\geq 0$. Define the variant of the Morrey-Campanato class $\mathcal{\bar{L}}^{p,\lambda}(\Omega)$ with
$$\|f\|^{p}_{\mathcal{\bar{L}}^{p,\lambda}(\Omega)}:=\sup_{x_{0}\in \Omega,\rho>0}\rho^{-\lambda}\int_{\Omega(x_{0},\rho)}\Big|f(x)-|f|_{\Omega(x_{0},r)}\Big|^pdx<\infty,$$
where $|f|_{\Omega(x_{0},r)}=\frac{1}{|\Omega(x_{0},r)|}\int_{\Omega(x_{0},r)}|f(y)|dy$. Some properties and characterizations will be shown.

\vspace{0.3cm}

On the other hand, The $BMO$ space is special case of Morrey-Campanato spaces, which is one of the important function spaces in harmonic analysis. For example, the singular integral operator maps from  $L^{\infty}(\mathbb{R}^n)$ to $BMO(\mathbb{R}^n)$ space and the dual theory of the classical Hardy space. Moreover, the foundational paper of Coifman,Rochberg and Weiss \cite{CRW1976} proved that the commutator
$$[b,T](f):=bT(f)-T(bf)$$
is bounded on some Lebesgue spaces if and only if $b$ belongs to $BMO(\mathbb{R}^n)$, where $T$ is the Riesz transforms. The theory was then generalized to several directions, and on the theory of the boundedness of commutators, many results show that $BMO$ function is the right set. Specially, in 2000, Bastero, Milman and Ruiz \cite{BMR} studied the class of functions for which the boundedness of commutator with the Hardy-Littlewood maximal function
$$M(f)(x):=\sup_{x\in Q}\frac{1}{|Q|}\int_{Q}|f(y)|dy.$$
They proved that $[b,M]$ is bounded on $L^{p}(\mathbb{R}^n)$ if and only if $b\in BMO(\mathbb{R}^n)$ with $b^{-}\in L^{\infty}(\mathbb{R}^n)$, where $b^{-}(x)=-\min\{b(x),0\}$ and $1<p<\infty$. In fact, $b\in BMO(\mathbb{R}^n)$ with $b^{-}\in L^{\infty}(\mathbb{R}^n)$ is equivalent to $b \in \mathcal{\bar{L}}^{p,n}(\Omega)$ with $0<p<\infty$. Applying the properties of $\mathcal{\bar{L}}^{p,\lambda}(\Omega)$, this paper will gives unified criterions on the necessity of bounded commutators of maximal operators.

\vspace{0.3cm}

This paper is organized as follows. In Section \ref{MORREY}, for $0<\lambda<n$ and $1\leq p<\infty$, we obtain the equivalence relationship between $\mathcal{\bar{L}}^{p,\lambda}(\Omega)$ and certain classical Morrey spaces, applying the methods used in most of the previous works. Section \ref{BMO} is concerned with BMO function with its negative part bounded. There we settle some technical theorems for the equivalent definitions and characterizations. In Section \ref{LIP}, some characterization of $\mathcal{\bar{L}}^{p,\lambda}(\Omega)$ will be given.
Finally, Section \ref{BANACH} contains some further results for the new classes of Morrey-Campanato function, and we establish a general criterion for the necessity of bounded commutators of maximal functions for general Banach spaces.
\section{The case $0\leq \lambda<n$.}\label{MORREY}

For $1\leq p<\infty$ and $0\leq \lambda<n$, it is obvious that there holds a continuous embedding $L^{p,\lambda}(\Omega)\hookrightarrow \mathcal{\bar{L}}^{p,\lambda}(\Omega)$.
Now, we show that $L^{p,\lambda}(\Omega)=\mathcal{\bar{L}}^{p,\lambda}(\Omega)$. The approach for this part is similar to that in \cite{C1964} but we need
to carefully use the properties of $|f|_{Q}$.

\begin{theorem}
For $1\leq p<\infty$ and $0\leq \lambda<n$, we have $L^{p,\lambda}(\Omega)\approx \mathcal{\bar{L}}^{p,\lambda}(\Omega)$.
\end{theorem}
\begin{proof}
We will freely use the inequality
$$(a+b)^{p}\leq 2^{p-1}(a^{p}+b^{p})$$
valid for every $p\geq 1$. Then
\begin{equation}\label{Morrey1}
\int_{\Omega(x_{0},\rho)}|f(x)-|f|_{\Omega(x_{0},\rho)}|^{p}dx\leq 2^{p-1}\bigg\{\int_{\Omega(x_{0},\rho)}|f(x)|^{p}dx+|\Omega(x_{0},\rho)|\big(|f|_{\Omega(x_{0},\rho)}\big)^{p}\bigg\}
\end{equation}
and by H\"{o}lder's inequality
\begin{equation}\label{Morrey2}
\big(|f|_{\Omega(x_{0},\rho)}\big)^{p}\leq \frac{1}{|\Omega(x_{0},\rho)|}\int_{\Omega(x_{0},\rho)}|f(x)|^pdx.
\end{equation}
Insert \eqref{Morrey2} to \eqref{Morrey1}, divide by $\rho^{\lambda}$ to obtain
$$\|f\|^{p}_{\mathcal{\bar{L}}^{p,\lambda}(\Omega)}\leq 2^{p}\|f\|^{p}_{L^{p,\lambda}(\Omega)},$$
thus concluding $L^{p,\lambda}(\Omega)\subset \mathcal{\bar{L}}^{p,\lambda}(\Omega)$.

On the other hand, it is easy to see that
\begin{equation}\label{Morrey3}
\int_{\Omega(x_{0},\rho)}|f|^{p}dx\leq 2^{p-1}\bigg\{\int_{\Omega(x_{0},\rho)}|f(x)-|f|_{\Omega(x_{0},\rho)}|^{p}dx+|\Omega(x_{0},\rho)|\big(|f|_{\Omega(x_{0},\rho)}\big)^{p}\bigg\}.
\end{equation}
We need to estimate the term $|f|_{\Omega(x_{0},\rho)}$ only. For $0<r<R$ and $x_{0},x\in \Omega$ we have
$$\big||f|_{\Omega(x_{0},R)}-|f|_{\Omega(x_{0},r)}\big|^{p}\leq 2^{p-1}\Big\{|f(x)-|f|_{\Omega(x_{0},R)}|^{p}+|f(x)-|f|_{\Omega(x_{0},r)}|^{p}\Big\}.$$
Integrating with respect to $x$ on  $\Omega(x_{0},r)$, we obtain
$$\big||f|_{\Omega(x_{0},R)}-|f|_{\Omega(x_{0},r)}\big|^{p}\leq
\frac{2^{p-1}}{Ar^n}\bigg(\int_{\Omega(x_{0},\rho)}\big|f(x)-|f|_{\Omega(x_{0},R)}\big|^{p}dx+\int_{\Omega(x_{0},r)}\big|f(x)-|f|_{\Omega(x_{0},r)}\big|^{p}dx\bigg),$$
thus
$$\big||f|_{\Omega(x_{0},R)}-|f|_{\Omega(x_{0},r)}\big|^{p}\leq C\frac{1}{r^n}(R^{\lambda}+r^{\lambda})\|f\|^{p}_{\bar{\mathcal{L}}^{p,\lambda}(\Omega)}
\leq C\frac{1}{r^n}R^{\lambda}\|f\|^{p}_{\bar{\mathcal{L}}^{p,\lambda}(\Omega)},$$
we arrive at
\begin{equation}\label{Morrey4}
\big||f|_{\Omega(x_{0},R)}-|f|_{\Omega(x_{0},r)}\big|\leq C\|f\|^{p}_{\bar{\mathcal{L}}^{p,\lambda}(\Omega)}R^{\frac{\lambda}{p}}r^{-\frac{n}{p}}.
\end{equation}
Set $R_{k}=\frac{R}{2^{k}}.$ The inequality \eqref{Morrey4} gives us that
\begin{equation}\label{Morrey5}
\big||f|_{\Omega(x_{0},R_{k})}-|f|_{\Omega(x_{0},R_{k+1})}\big|\leq C\|f\|_{\bar{\mathcal{L}}^{p,\lambda}(\Omega)}R^{\frac{\lambda-n}{p}}2^{\frac{k(n-\lambda)+n}{p}}.
\end{equation}
It follows from \eqref{Morrey5} that
\begin{equation}\label{Morrey6}
\big||f|_{\Omega(x_{0},R)}-|f|_{\Omega(x_{0},R_{h+1})}\big|\leq C\|f\|_{\bar{\mathcal{L}}^{p,\lambda}(\Omega)}R_{h+1}^{\frac{\lambda-n}{p}}.
\end{equation}
Choose $h$ and $R$ such that $diam \Omega\leq R\leq 2diam \Omega$ and $R_{h+1}=\rho$, we get
\begin{eqnarray*}
\begin{aligned}
|f|^{p}_{\Omega(x_{0},\rho)}&\leq 2^{p-1}\Big(|f|^{p}_{\Omega(x_{0},R)}+\big||f|_{\Omega(x_{0},\rho)}-|f|_{\Omega(x_{0},R)}\big|^{p}\Big)\\
&\leq 2^{p-1}\Big(|f|^{p}_{\Omega(x_{0},R)}+\rho^{\lambda-n}\|f\|^{p}_{\bar{\mathcal{L}}^{p,\lambda}(\Omega)}\Big).
\end{aligned}
\end{eqnarray*}
Combining with \eqref{Morrey3} and the fact that
$$\rho^{n-\lambda}|f|^{p}_{\Omega(x_{0},R)}\leq C|f|^{p}_{\Omega(x_{0},R)}\leq C\|f\|^p_{L^{p}(\Omega)},$$
we conclude that
\begin{eqnarray*}
\begin{aligned}
\frac{1}{\rho^{\lambda}}\int_{\Omega(x_{0},\rho)}|f(x)|^{p}dx\leq C\Big(\|f\|^{p}_{\bar{\mathcal{L}}^{p,\lambda}(\Omega)}+|f|_{\Omega(x_{0},R)}\Big)\leq C\|f\|^p_{\bar{\mathcal{L}}^{p,\lambda}(\Omega)}.
\end{aligned}
\end{eqnarray*}
We have that $L^{p,\lambda}(\Omega)\supseteq \mathcal{\bar{L}}^{p,\lambda}(\Omega)$ and the proof is complete.
\end{proof}

\section{The case $\lambda=n$.}\label{BMO}
The notion of functions of bounded mean oscillation was introduced and
studied by John and Nirenberg \cite{JN1961} in connection with the work of
John on quasi-isometric maps and of Moser on Harnack inequality.
Let $Q_{0}$ be a cube in $\mathbb{R}^n$. We say that a function $f\in L^{1}(Q_{0})$ belongs to the space of functions with bounded mean oscillation $BMO(Q_{0})$ if
\begin{equation}\label{bmo}
|f|_{*}:=\sup\frac{1}{|Q|}\int_{Q}|f-f_{Q}|dx<+\infty,
\end{equation}
where the supremum it taken over all the cubes $Q\subset Q_{0}$.

Commonly $BMO$ is defined in the whole of $\mathbb{R}^n$ by requiring $f\in L^{1}_{loc}(\mathbb{R}^n)$ and the supremum in \eqref{bmo} to be taken over all cubes in $\mathbb{R}^n$.
It is easy to see that $BMO(Q_{0})= \mathcal{L}^{1,n}(Q_{0})$. In this section, we shall in fact see later that the function $f\in \bar{\mathcal{L}}^{p,n}(Q_{0})$ for all $p, 0<p<\infty$ if and only if $f\in BMO(Q_{0})$ with $f^{-}\in L^{\infty}(Q_{0})$.

\subsection{John-Nirenberg Lemma}

To obtain the desired results, we need the following John-Nirenberg inequality for $\mathcal{\bar{L}}^{p,n}(Q_{0})$ with $0<p\leq 1$.

\begin{theorem}\label{JN}
Let $0<p\leq 1$. There are constants $c_{1},c_{2}>0$, depending only $n$, such that
\begin{equation}\label{eqJN}
\Big|\{x\in Q: \big|f(x)-|f|_{Q}\big|>t\}\Big|\leq c_{1}\exp\Big(-\frac{c_{2}t}{\|f\|_{\bar{\mathcal{L}}^{p,n}(Q_{0})}}\Big)|Q|
\end{equation}
for all $Q\subset Q_{0}$ with sides parallel to those of $Q_{0}$, all $f\in \bar{\mathcal{L}}^{p,n}(Q_{0})$ and all $t>0$ .
\end{theorem}
\begin{proof}
By $f\in \bar{\mathcal{L}}^{p,n}(Q_{0})\Rightarrow f\in \bar{\mathcal{L}}^{p,n}(Q)$, it is enough to prove \eqref{eqJN} for $Q=Q_{0}$ only. Let $\|f\|_{\bar{\mathcal{L}}^{p,n}(Q_{0})}=1$.
Take $$\alpha>1\geq \frac{1}{|Q_{0}|}\int_{Q_{0}}\big|f(x)-|f|_{Q_{0}}\big|^{p}dx$$
and applying the Calder\`{o}n-Zygmund decomposition with $f=\big|f-|f|_{Q_{0}}\big|^{p}$ and parameter $\alpha$, we can obtain a sequence $\{Q^{1}_{k}\}_{k\in K_{1}}$ satisfy
\begin{equation}\label{eqJN1}
\alpha<\frac{1}{|Q^{1}_{k}|}\int_{Q^{1}_{k}}\big|f(x)-|f|_{Q_{0}}\big|^{p}dx\leq 2^{n}\alpha, \quad \text{for all} \quad   k\in K_{1},
\end{equation}
and
\begin{equation}\label{eqJN2}
\big|f-|f|_{Q_{0}}\big|^{p}\leq \alpha \qquad  \text{a.e. \ on} \quad Q_{0}\backslash \bigcup_{k\in K_{1}}Q^{1}_{k}.
\end{equation}
By \eqref{eqJN1} we have
\begin{eqnarray*}\label{eqJN3}
\begin{aligned}
\big||f|_{Q^{1}_{k}}-|f|_{Q_{0}}\big|^{p}&=\frac{1}{|Q^{1}_{k}|}\int_{Q^{1}_{k}}\big||f|_{Q^{1}_{k}}-|f|_{Q_{0}}\big|^{p}dx\\
&\leq \frac{1}{|Q^{1}_{k}|}\int_{Q^{1}_{k}}\big|f(x)-|f|_{Q^1_{k}}\big|^{p}dx+ \frac{1}{|Q^{1}_{k}|}\int_{Q^{1}_{k}}\big|f(x)-|f|_{Q_{0}}\big|^{p}dx\\
&\leq 1+2^{n}\alpha\leq 2^{n+1}\alpha,
\end{aligned}
\end{eqnarray*}
and
\begin{equation}\label{eqJN4}
\sum_{k\in K_{1}}|Q^{1}_{k}|\leq \frac{1}{\alpha}\sum_{k\in K_{1}}\int_{Q^{1}_{k}}\big|f(x)-|f|_{Q_{0}}\big|^{p}dx
\leq \frac{1}{\alpha}\int_{Q_{0}}\big|f(x)-|f|_{Q_{0}}\big|^pdx\leq \frac{1}{\alpha}|Q_{0}|.
\end{equation}

We can apply again the Calder\`{o}n-Zygmund decomposition with $Q^1_{k}$, $f=\big|u-|u|_{Q^1_k}\big|$ and parameter $\alpha$, we also can find a sequence of cubes $\{Q^1_{k,j}\}_{j\in J(k)}$ such that
\begin{equation}\label{eqJN5}
\alpha<\frac{1}{|Q^{1}_{k,j}|}\int_{Q^{1}_{k,j}}\big|f(x)-|f|_{Q_{0}}\big|^{p}dx\leq 2^{n}\alpha, \quad \text{for all} \quad   j\in J_{k},
\end{equation}
and
\begin{equation}\label{eqJN6}
\big|f-|f|_{Q_{0}}\big|^{p}\leq \alpha \qquad  \text{a.e. \ on} \quad Q^1_{k}\backslash \bigcup_{j\in J_{k}}Q^{1}_{k,j}.
\end{equation}
Write
$$\{Q_{k,j}\}_{j\in J(k),k\in K_{1}}=\{Q^{2}_{k}\}_{k\in K_{2}}.$$
From \eqref{eqJN3} and \eqref{eqJN6}, for $x\in Q_{0}\backslash \bigcup_{k\in K_{2}}Q^{2}_{k}$, there is a unique index $k=k(x)\in K_{1}$ such that $x\in Q^1_k$, we get
\begin{equation*}\label{eqJN7}
\big|f(x)-|f|_{Q_{0}}\big|^{p}\leq \big|f(x)-|f|_{Q^1_{k}}\big|^{p}+\big||f|_{Q_{0}}-|f|_{Q^1_{k}}\big|^{p}\leq 2^{n+2}\alpha.
\end{equation*}
Moreover, it follows from \eqref{eqJN4} that
$$\sum_{j\in K_{2}}|Q_{k}^{2}|\leq \frac{1}{\alpha}\sum_{k\in K_{1}}\int_{Q^1_{k}}\big|f(x)-|f|_{Q^1_{k}}\big|^{p}dx
\leq \frac{1}{\alpha}\sum_{k\in K_{1}}|Q^1_{k}|\leq \frac{1}{\alpha^2}|Q_{0}|.$$
Repeating this procedure inductively, for every $k\in \mathbb{N}$ we can obtain a sequence of cubes $\{Q^i_{k}\}_{k\in K_i}$ such that
\begin{equation*}\label{eqJN8}
\big|f-|f|_{Q_{0}}\big|^{p}\leq i 2^{n+1}\alpha \qquad  \text{a.e. \ on} \quad Q_0\backslash \bigcup_{k\in K_{i}}Q^{i}_{k}.
\end{equation*}
and
\begin{equation*}\label{eqJN9}
\sum_{k\in K_{i}}|Q_{k}^{i}|\leq \frac{1}{\alpha^i}|Q_{0}|.
\end{equation*}

For $t>2^{n+1}\alpha$, choose $i\in\mathbb{N}$ in such a way that $i2^{n+1}\alpha<t\leq (i+1)2^{n+1}\alpha$. We set $c_{1}=\alpha$ and $c_{2}=\frac{\log \alpha}{2^{n+1}\alpha}$, then
\begin{eqnarray*}
\begin{aligned}
\Big|\{x\in Q_{0}: \big|f(x)-|f|_{Q_{0}}\big|>t\}\Big|&\leq \Big|\{x\in Q_{0}: \big|f(x)-|f|_{Q_{0}}\big|>i2^{n+1}\alpha\}\Big|\\
&\leq \sum_{k\in K_{i}}|Q^i_{k}|\leq \frac{1}{\alpha^i}|Q_{0}|\\
&\leq c_{1}e^{-c_{2}t}|Q_{0}|.
\end{aligned}
\end{eqnarray*}
For $0<t\leq 2^{n+1}\alpha$, the result is obtained directly. This implies the desired conclusion and completes the proof of the Theorem \ref{JN}.
\end{proof}

\begin{theorem}\label{p<1}
Let $0<p<1$ and $Q_{0}$ be a cube in $\mathbb{R}^n$. Then
$$\mathcal{\bar{L}}^{p,n}(Q_{0})=\mathcal{\bar{L}}^{1,n}(Q_{0})$$
with equivalence of the corresponding norms.
\end{theorem}
\begin{proof}
Using \eqref{JN}, for any $Q\subset Q_{0}$ and $f\in \mathcal{\bar{L}}^{p,n}(Q_{0})$,
\begin{eqnarray*}
\begin{aligned}
\int_{Q}|f(x)-|f|_{Q}|dx&=\int_{0}^{\infty}\Big|\{x\in Q_{0}: \big|f(x)-|f|_{Q}\big|>t\}\Big|dt\\
&\leq c_{1}\int_{0}^{\infty}\exp{\big(\frac{-c_{2}t}{\|f\|_{\mathcal{\bar{L}}^{p,n}(Q_{0})}}}|Q|dt\\
&\leq C\|f\|_{\mathcal{\bar{L}}^{p,n}(Q_{0})}|Q|,
\end{aligned}
\end{eqnarray*}
then $f\in \mathcal{\bar{L}}^{1,n}(Q_{0})$.

Conversely, it is immediately that $\|f\|_{\mathcal{\bar{L}}^{p,n}(Q_{0})}\leq \|f\|_{\mathcal{\bar{L}}^{1,n}(Q_{0})}$ by H\"{o}lder inequality.
\end{proof}

Similarly, we also have the result for $1<p<\infty$ as follows.
\begin{theorem}\label{p>1}
Let $1<p<\infty$ and $Q_{0}$ be a cube in $\mathbb{R}^n$. Then
$$\mathcal{\bar{L}}^{p,n}(Q_{0})=\mathcal{\bar{L}}^{1,n}(Q_{0})$$
with equivalence of the corresponding norms.
\end{theorem}

In addition, \eqref{eqJN} is in fact equivalent to $f$ being a $\mathcal{\bar{L}}^{1,n}(Q_{0})$ function.
\begin{theorem}\label{Cha0}
The following facts are equivalent:
\begin{itemize}
  \item [(i)] $f\in \mathcal{\bar{L}}^{1,n}(Q_{0})$;
  \item [(ii)] there are $c_{1},c_{2}$ such that for all $Q\subset Q_{0}$ and $t>0$,
  $$\Big|\{x\in Q: \big|f(x)-|f|_{Q}\big|>t\}\Big|\leq c_{1}e^{-c_{2}t}|Q|;$$
  \item [(iii)] there are $c_{3},c_{4}$ such that for all $Q\subset Q_{0}$
  $$\frac{1}{|Q|}\int_{Q}e^{c_{4}\big|f(x)-|f|_{Q}\big|}-1dx\leq c_{3}.$$
\end{itemize}
\end{theorem}
\begin{proof}
John-Nirenberg lemma yields the fact that $(i)\Rightarrow (ii)$.

$(ii)\Rightarrow (iii)$. Set $c_{4}:=\frac{c_{2}}{2}$ and $s:= e^{c_{4}t}$. We get
\begin{eqnarray*}
\begin{aligned}
\frac{1}{|Q|}\int_{Q}e^{c_{4}\big|f(x)-|f|_{Q}\big|}-1dx
&=\int_{1}^{\infty}\Big|\{x\in Q: e^{c_{4}\big|f(x)-|f|_{Q}\big|}>s\}\Big|ds\\
&=\int_{0}^{\infty}c_{4}e^{c_{4}t}\Big|\{x\in Q:\big|f(x)-|f|_{Q}\big|>t\}\Big|dt\\
&\leq c_{1}|Q|\int_{0}^{\infty}c_{4}e^{c_{4}t}e^{-c_{2}t}dt\\
&=c_{1}|Q|\int_{0}^{\infty}\frac{c_{4}}{2}e^{-\frac{c_{2}}{2}t}dt=c_{1}|Q|.
\end{aligned}
\end{eqnarray*}

$(iii)\Rightarrow (i)$. As $t\leq e^{t}-1$, hence $\big|f-|f|_{Q}\big|\leq \frac{1}{c_{4}}e^{c_{4}|f-|f|_{Q}|}-1$, then
$$\frac{1}{|Q|}\int_{Q}|f(x)-|f|_{Q}|dx\leq \frac{1}{c_{4}}\frac{1}{|Q|}\int_{Q}\Big(e^{c_{4}|f-|f|_{Q}|}-1\Big)dx\leq \frac{c_{3}}{c_{4}}.$$

Thus we complete the proof of Theorem \ref{Cha0}.
\end{proof}

\subsection{Some equivalent definitions of $\mathcal{\bar{L}}^{1,n}(Q_{0})$ function.}

Next, we give some equivalent definitions of $\mathcal{\bar{L}}^{1,n}(Q_{0})$ function.
\begin{theorem}\label{def}
Let $Q_{0}$ be a cube in $\mathbb{R}^n$. Then the following statements are equivalent:
\begin{itemize}
  \item [(i)] $f\in \mathcal{\bar{L}}^{1,n}(Q_{0})$,
  \item [(ii)] $f\in BMO(Q_{0})$ with $f^{-}\in L^{\infty}(Q_{0})$;
  \item [(iii)] For any $Q\subset Q_{0}$, there is a constant $C$ such that
  $$\frac{1}{|Q|}\int_{Q}\big||f(x)|-f_{Q}\big|dx\leq C.$$
\end{itemize}
\end{theorem}
\begin{proof}
$(i)\Rightarrow (ii).$ We write $E_{Q}=\{x\in Q: f(x)\geq f_{Q}\}$ and $F_{Q}=Q\backslash E_{Q}$. From the fact
$$\int_{Q}\big(f(x)-f_{Q}\big)dx=0,$$
it follows that
\begin{eqnarray*}
\begin{aligned}
\frac{1}{|Q|}\int_{Q}\big|f(x)-f_{Q}\big|dx&= \frac{2}{|Q|}\int_{F_{Q}}\big(f_{Q}-f(x)\big)dx\\
&\leq  \frac{2}{|Q|}\int_{F_{Q}}\big(|f|_{Q}-f(x)\big)dx \\
&\leq  \frac{2}{|Q|}\int_{Q}\big||f|_{Q}-f(x)\big|dx\leq 2\|f\|_{\mathcal{\bar{L}}^{1,n}(Q_{0})}.
\end{aligned}
\end{eqnarray*}
Meanwhile,
\begin{eqnarray*}
\begin{aligned}
\frac{2}{|Q|}\int_{Q}f^-(x)dx&= |f|_{Q}-f_{Q}\leq \frac{1}{|Q|}\int_{Q}\big|f(x)-|f|_{Q}\big|dx\leq \|f\|_{\mathcal{\bar{L}}^{1,n}(Q_{0})}.
\end{aligned}
\end{eqnarray*}
which implies that $f^-\in L^{\infty}(Q_{0})$ by Lebesgue different theorem.

$(ii)\Rightarrow (iii).$ Let $f\in BMO(Q_{0})$ with $f^{-}\in L^{\infty}(Q_{0})$. For any $Q\subset Q_{0}$,
\begin{eqnarray*}
\begin{aligned}
\frac{1}{|Q|}\int_{Q}\big||f(x)|-f_{Q}\big|dx&\leq \frac{1}{|Q|}\int_{Q}|f(x)-f_{Q}|dx+(f^-)_{Q}\\
&\leq |f|_{*}+\|f^{-}\|_{L^{\infty}(Q_{0})}.
\end{aligned}
\end{eqnarray*}

$(iii)\Rightarrow (i).$ Using the fact
$$\frac{2}{|Q|}\int_{Q}f^-(x)dx= |f|_{Q}-f_{Q},$$
we arrive at
\begin{eqnarray*}
\begin{aligned}
&\frac{1}{|Q|}\int_{Q}\big|f(x)-|f|_{Q}\big|dx\\
&\leq \frac{1}{|Q|}\int_{Q}\big|f(x)-|f(x)|\big|dx+\big(|f|_{Q}-f_{Q}\big)+\frac{1}{|Q|}\int_{Q}\big|f_{Q}-|f(x)|\big|dx\\
&\leq \frac{3}{|Q|}\int_{Q}\big||f(x)|-f_{Q}\big|dx.
\end{aligned}
\end{eqnarray*}
Therefore, we complete the proof of Theorem \ref{def}.
\end{proof}

On the other hand, it is well known that $f\in BMO(Q_{0})$ if and only if
\begin{equation}\label{eqc1}
\sup_{Q\subset Q_{0}}\inf_{c\in \mathbb{R}^n}\rho^{-n}\int_{Q}|u(x)-c|dx<\infty.
\end{equation}
A similar conclusion can be obtained as follows.
\begin{theorem}\label{Pc1}
Let $Q_{0}$ be an $n$-dimensional cube in $\mathbb{R}^n$. Then $f\in \mathcal{\bar{L}}^{1,n}(Q_{0})$ if and only if
$$\sup_{Q\subset Q_{0}}\inf_{c\geq 0}\rho^{-n}\int_{Q}|f(x)-c|dx<\infty.$$
\end{theorem}
\begin{proof}
$(\Rightarrow)$ is obviously.

$(\Leftarrow)$. For each cube $Q$ let $c_{Q}\geq 0$ be the value which minimizes $\displaystyle{\int_{Q}|f(x)-c|dx}$ with $c\geq 0$, then
$$\frac{1}{|Q|}\int_{Q}|f(x)-|f|_{Q}|dx\leq \frac{1}{|Q|}\int_{Q}|f(x)-c_{Q}|dx+\frac{1}{|Q|}\int_{Q}|c_{Q}-|f|_{Q}|dx.$$
By $c_{Q}\geq 0$, we have
$$
||f|_{Q}-c_{Q}|\leq|f_{Q}-c_{Q}|\leq\frac{1}{|Q|}\int_{Q}|f(x)-c_{Q}|dx.
$$
Therefore
$$\frac{1}{|Q|}\int_{Q}|f(x)-|f|_{Q}|dx\leq C$$
and $f\in \mathcal{\bar{L}}^{1,n}(Q_{0})$.
\end{proof}

In fact, we can obtain replace \eqref{eqc1} with the following \eqref{eqcp}, for $0<p<\infty$
\begin{equation}\label{eqcp}
\|f\|^p_{\mathcal{\bar{L}}_{*}^{p,n}(Q_{0})}=\sup_{Q\subset Q_{0}}\inf_{c\geq 0}\rho^{-n}\int_{Q}|f(x)-c|^pdx<\infty,
\end{equation}
Now, we establish a version of John-Nirenberg inequality suitable for $\mathcal{\bar{L}}_{*}^{p,n}(Q_{0})$ with $0<p<1$.

\begin{theorem}\label{JNcp}
Let $0<p<1$ and $\|f\|_{\mathcal{\bar{L}}_{*}^{p,n}(Q_{0})}=1$ and for each cube $Q$ let $c_{Q}$ be the positive constant which minimizes $\displaystyle{\int_{Q}|f(x)-c|^{p}dx}$. Then
$$\Big|\big\{x\in Q: |f(x)-c_{Q}|>t\big\}\Big|\leq c_{1}e^{-c_{2}t}|Q|,$$
where $c_{1}$ and $c_{2}$ are positive constants.
\end{theorem}
\begin{proof}
Take any cube $Q$, write $E_{Q}=\{x\in Q: |f(x)-c_{Q}|>t\}$. Then
\begin{eqnarray*}
|E_{Q}|&\leq& \int_{E_{Q}}\frac{|f(x)-c_{Q}|^{p}}{t^{p}}dx\\
&\leq&\frac{1}{t^{p}}\int_{Q}|f(x)-c_{Q}|^{p}dx\\
&\leq&\frac{1}{t^{p}}|Q|.
\end{eqnarray*}
Write $F_{1}(t)=\frac{1}{t^{p}}$, then
$$|E_{Q}|\leq F_{1}(t)|Q|.$$

Let $s>1$ and $t\in (0,\infty)$ such that $2^{\frac{n+1}{p}}s\leq t$. Fix a cube $Q$, there is a Calderon-Zygmund decomposition to $|f(x)-c_{Q}|^{p}$ of disjoint cubes $\{Q_{j}\}$ such that $Q_{j}\subset Q$ and\\

~~$(i)~~ \displaystyle{s^{p}<\frac{1}{|Q_{j}|}\int_{Q_{j}}|f(x)-c_{Q}|^{p}dx\leq 2^{n}s^{p}},$\\

~~$(ii)~~\displaystyle{|f(x)-c_{Q}|\leq s}$ for $x\in Q\backslash\bigcup_{j}Q_{j}$.
\vspace{0.3cm}

Notice that
$$\int_{Q_{j}}|f(y)-c_{Q_{j}}|^{p}dy\leq \int_{Q_{j}}|f(y)-c_{Q}|^{p}dy.$$
Therefore, by $(i)$, we have
\begin{eqnarray*}
|c_{Q_{j}}-c_{Q}|^{p}
&=&\frac{1}{|Q_{j}|}\int_{Q_{j}}|c_{Q_{j}}-c_{Q}|^{p}dy\\
&\leq&\frac{1}{|Q_{j}|}\int_{Q_{j}}|f(y)-c_{Q_{j}}|^{p}dy+\frac{1}{|Q_{j}|}\int_{Q_{j}}|f(y)-c_{Q}|^{p}dy\\
&\leq&\frac{2}{|Q_{j}|}\int_{Q_{j}}|f(y)-c_{Q}|^{p}dy\\
&\leq&2^{n+1}s^{p}.
\end{eqnarray*}

Since $2^{\frac{n+1}{p}}s\leq t$, then $E_{Q}\subset \bigcup_{j}Q_{j}$ and
\begin{eqnarray*}
|E_{Q_{0}}|&=&\sum_{j}\big|\{x\in Q_{j}:|f(x)-c_{Q}|>t\}\big|\\
&\leq&\sum_{j}\big|\{x\in Q_{j}:|f(x)-c_{Q_{j}}|+|c_{Q_{j}}-c_{Q}|>t\}\big)\\
&\leq&\sum_{j}\big|\{x\in Q_{j}:|f(x)-c_{Q_{j}}|>t-2^{\frac{n+1}{p}}s\}\big|\\
&\leq&\sum_{j}F_{1}(t-2^{\frac{n+1}{p}}s)\cdot |Q_{j}|\\
&\leq&F_{1}(t-2^{\frac{n+1}{p}}s)
\sum_{j}\frac{1}{s^{p}}\int_{Q_{j}}|f(x)-c_{Q}|^{p}dx\\
&\leq&\frac{F_{1}(t-2^{\frac{n+1}{p}}s)}{s^{p}}\int_{Q_{0}}|f(x)-c_{Q}|^{p}dx\\
&\leq&\frac{F_{1}(t-2^{\frac{n+1}{p}}s)}{s^{p}}|Q|.
\end{eqnarray*}
Let
$$F_{2}(t)= \frac{F_{1}(t-2^{\frac{n+1}{p}}s)}{s^{p}}.$$
Continue this process indefinitely, we obtain for any $k\geq 2$,
$$F_{k}(t)= \frac{F_{k-1}(t-2^{n+1}s^{p})}{s^{p}}.$$
and
$$|E_{Q}|\leq F_{k}(t)|Q|.$$

We fix a constant $t>0$. If
$$k\cdot2^{\frac{n+1}{p}}s<t\leq (k+1)\cdot2^{\frac{n+1}{p}}s.$$
for some $k\geq 1$, thus
\begin{eqnarray*}
|E_{Q}|&\leq& \big|\{x\in Q_{0}:|f(x)-c_{Q}|>t\}\big|\\
&\leq&\big|\{x\in Q:|f(x)-c_{Q}|>k\cdot2^{\frac{n+1}{p}}s\}\big|\\
&\leq&F_{k}(k\cdot2^{\frac{n+1}{p}}s)|Q|\\
&=&\frac{F_{1}(2^{\frac{n+1}{p}}s)}{s^{(k-1)p}}|Q|\\
&=&\frac{1}{2^{n+1}s^{kp}}|Q|\\
&\leq&\frac{e^{-kp\log s}}{2^{n+1}}|Q|\\
&\leq&\frac{e}{2^{n+1}}e^{-2^{-\frac{n+1}{p}}t}|Q|.
\end{eqnarray*}
Since $-k\leq 1-\frac{t}{2^{\frac{n+1}{p}}s}$. If $t\leq 2^{\frac{n+1}{p}}s$, then use the trivial estimate
$$|E_{Q}|\leq |Q|\leq e^{-t}e^{2^{\frac{n+1}{p}}s}|Q|.$$
Recall that $s$ is any real number greater that 1. Choosing $s=e$, this yields
$$\Big|\big\{x\in Q: |f(x)-c_{Q}|>t\big\}\Big|\leq c_{1}e^{-c_{2}t}|Q|,$$
for some positive constants $c_{1}$ and $c_{2}$, which proves the inequality of the Proposition \ref{JNcp}.
\end{proof}

From  Theorems \ref{Pc1} and \ref{JNcp}, it is immediately that
\begin{theorem}\label{Pp}
Let $0<p<\infty$ and $Q_{0}$ be an $n$-dimensional cube in $\mathbb{R}^n$. Then $f\in \mathcal{\bar{L}}^{1,n}(Q_{0})$ if and only if
$$\sup_{Q\subset Q_{0}}\inf_{c\geq 0}\rho^{-n}\int_{Q}|f(x)-c|^pdx<\infty.$$
\end{theorem}

\begin{remark}
It is worth remarking that
\begin{itemize}
  \item [(a)] $f\in \mathcal{\bar{L}}^{1,n}(Q_{0})$ if and only if for every $Q\subset Q_{0}$ there is a constant $c_{Q}\geq 0$ such that
$$\sup_{Q}\frac{1}{|Q|}\int_{Q}|f(x)-c_{Q}|dx<\infty.$$
Indeed for $x\in Q$,
\begin{eqnarray*}
\begin{aligned}
\big|f(x)-|f|_{Q}\big|&\leq |f(x)-c_{Q}|+\big|c_{Q}-|f|_{Q}\big|\\
&\leq |f(x)-c_{Q}|+\frac{1}{|Q|}\int_{Q}|f(y)-c_{Q}|dy,
\end{aligned}
\end{eqnarray*}
and averaging over $Q$ we get
\begin{eqnarray*}
\begin{aligned}
\frac{1}{|Q|}\int_{Q}|f(x)-|f|_{Q}|dx&\leq \frac{2}{|Q|}\int_{Q}|f(x)-c_{Q}|dx.
\end{aligned}
\end{eqnarray*}

  \item [(b)]  $f\in \mathcal{\bar{L}}^{1,n}(Q_{0})$ if and only if for every $Q\subset Q_{0}$ there is a constant $c_{Q}\leq 0$ such that
$$\sup_{Q}\frac{1}{|Q|}\int_{Q}\big||f(x)|-c_{Q}\big|dx<\infty.$$
Indeed for $x\in Q$
\begin{eqnarray*}
\begin{aligned}
\big||f(x)|-f_{Q}\big|&\leq \big||f(x)|-c_{Q}\big|+\big|c_{Q}-f_{Q}\big|\\
&\leq |f(x)-c_{Q}|+\frac{1}{|Q|}\int_{Q}|f(y)-c_{Q}|dy\\
&\leq |f(x)-c_{Q}|+\frac{1}{|Q|}\int_{Q}\big||f(y)|-c_{Q}\big|dy,
\end{aligned}
\end{eqnarray*}
and averaging over $Q$ we get
\begin{eqnarray*}
\begin{aligned}
\frac{1}{|Q|}\int_{Q}\big||f(x)|-f_{Q}|dx&\leq \frac{2}{|Q|}\int_{Q}\big||f(x)|-c_{Q}\big|dx.
\end{aligned}
\end{eqnarray*}
\end{itemize}
\end{remark}

\subsection{Characterizations of $\mathcal{\bar{L}}^{1,n}(\mathbb{R}^n)$ function associated to maximal functions.}

In 2000, Bastero, Milman and Ruiz \cite{BMR} showed that $f\in BMO(\mathbb{R}^n)$ with $f^{-}\in L^\infty(\mathbb{R}^n)$ if and only if
\begin{equation}\label{eq1.1}
\sup_{Q}\bigg(\frac{1}{|Q|}\int_{Q}|f(x)-M_{Q}(f)(x)|^{p}dx\bigg)^{1/p}<\infty,
\end{equation}
where $1\leq p<\infty$ and
$$\displaystyle {M_{Q}(f)(x)=\sup_{Q\supseteq Q'\ni x}\frac{1}{|Q'|}\int_{Q'}|f(y)|dy}.$$

Later, Zhang and Wu obtained similar results for the fractional maximal function in \cite{ZW2009} and extended the above results to variable exponent Lebesgue
spaces in \cite{ZW2014} and \cite{ZW20142}. For $0<\alpha<n$, the fractional maximal function is defined by
$$M_{\alpha}(f)(x)=\sup_{x\in Q}\frac{1}{|Q|^{1-\alpha/n}}\int_{Q}|f(y)|dy,$$
It is proved that for $1<p,q<\infty$ and $1/p-1/q=\alpha/n$, the commutator $[b,M_{\alpha}]$ is bounded from $L^{p}(\mathbb{R}^n)$ to $L^{q}(\mathbb{R}^n)$ if and only if
 \begin{equation}\label{eq1.1}
\sup_{Q}\bigg(\frac{1}{|Q|}\int_{Q}|f(x)-|Q|^{-\alpha,n}M_{\alpha,Q}(f)(x)|^{q}dx\bigg)^{1/q}<\infty,
\end{equation}
where
$$\displaystyle {M_{\alpha,Q}(f)(x)=\sup_{Q\supseteq Q'\ni x}\frac{1}{|Q'|^{1-\alpha/n}}\int_{Q'}|f(y)|dy}.$$

The above results show that the \eqref{eq1.1} is equivalent in the sense of norm, when $1\leq p<\infty$. It is a natural question: can we establish the John-Nirenberg inequality suitable for the function $f\in BMO(\mathbb{R}^n)$ with $f^{-}\in L^{\infty}(\mathbb{R}^n)$? In \cite{HD}, we solved the problem as follows.
\begin{theorem}\label{JNM1} \cite{HD}
Suppose $f\in BMO(\mathbb{R}^n)$ with $f^{-}\in L^{\infty}(\mathbb{R}^n)$, then for any cube $Q$ and $t>0$, we have
$$\Big|\big\{x\in Q: |f(x)-M_{Q}(f)(x)|>t\big\}\Big|\leq a_{1}e^{-a_{2}t}|Q|.$$
where $a_{1}$ and $a_{2}$ are positive constants.
\end{theorem}

Meanwhile, we now establish the John-Nirenberg inequality associated to fractional maximal functions.

\begin{theorem}\label{JNMalpha}
Suppose $f\in BMO(\mathbb{R}^n)$ with $f^{-}\in L^{\infty}(\mathbb{R}^n)$, then for any cube $Q$ and $t>0$, we have
$$\Big|\big\{x\in Q: |f(x)-|Q|^{-\alpha/n}M_{\alpha,Q}(f)(x)|>t\big\}\Big|\leq a_{1}e^{-a_{2}t}|Q|.$$
where $a_{1}$ and $a_{2}$ are positive constants.
\end{theorem}
\begin{proof}
For any cube $Q$ and $x\in Q$, it follows from the definitions of $M_{Q}$ and $M_{\alpha,Q}$ that
$$|f|_{Q}\leq |Q|^{-\alpha/n}M_{\alpha,Q}(f)(x)\leq M_{Q}(f)(x).$$
Set $E:=\{y\in Q: f(y)\geq |f|_{Q}\}$ and $F:=Q\backslash E$, then
\begin{eqnarray*}
\begin{aligned}
&\Big|\big\{x\in Q: |f(x)-|Q|^{-\alpha/n}M_{\alpha,Q}(f)(x)|>t\big\}\Big|\\
&= \Big|\big\{x\in E: f(x)-|Q|^{-\alpha/n}M_{\alpha,Q}(f)(x)>t\big\}\Big|+\Big|\big\{x\in F: |Q|^{-\alpha/n}M_{\alpha,Q}(f)(x)-f(x)>t\big\}\Big|\\
&\leq \Big|\big\{x\in E: f(x)-|f|_{Q}>t\big\}\Big|+\Big|\big\{x\in F: M_{Q}(f)(x)-f(x)>t\big\}\Big|\\
&\leq a_{1}e^{-a_{2}t}|Q|.
\end{aligned}
\end{eqnarray*}
We complete the proof of Theorem \ref{JNMalpha}.
\end{proof}

\section{The case $n< \lambda\leq n+p$.}\label{LIP}

In the theory of PDEs one encounters two types of a priori estimates: $L^{p}$ norm estimates and Schauder estimates (estimates in the space of the H\"{o}lder continuous function). An attempt to achieve
this is provided by the integral characterization of the H\"{o}lder continue functions.

\subsection{A characterization of nonnegative H\"{o}lder continuous functions}

 In this section, we give the equivalent definitions of nonnegative H\"{o}lder continuous functions.

\begin{theorem}\label{Cmain}
Let $1\leq p<\infty$ and $0<\beta\leq 1$. For the function $f\in L_{loc}(\Omega)$, the following three statements are equivalent:
\begin{enumerate}
\item [\rm(1)] $f\in C^{0,\beta}(\Omega)$ and $f\geq 0$.
\item [\rm(2)] There exists a constant $C_{1}$ such that
$$\big|f(x)-|f(y)|\big|\leq C_{1}|x-y|^{\beta}$$
for almost every $x$ and $y$.
\item [\rm(3)] There exists a constant $C_{2}$ such that for any $x_{0}\in \Omega$ and $0<\rho<diam \Omega$
$$\Bigg(\frac{1}{|\Omega(x_{0},\rho)|}\int_{\Omega(x_{0},\rho)}\big|f(x)-|f|_{\Omega(x_{0},r)}\big|^pdx\Bigg)^{1/p}\leq C_{2}\rho^{\beta}.$$
\end{enumerate}
\end{theorem}
\begin{proof}
${\rm (1)}\Longrightarrow {\rm (2)}$. Assume $f\in C^{0,\alpha}(\Omega)$ and $f\geq 0$. For $x,y\in \Omega$, we conclude that
\begin{equation*}
\big|f(x)-|f(y)|\big|=\big|f(x)-f(y)\big|\leq \|f\|_{C^{0,\beta}(\Omega)}.
\end{equation*}

${\rm (2)}\Longrightarrow {\rm (3)}$. If $x,y\in \Omega(x_{0},\rho)$, from the fact that $\big|f(x)-|f(y)|\big|\leq C_{1}|x-y|^{\beta}$ we have
$$\big|f(x)-|f(y)|\big|\leq C_{1}\rho^{\alpha},$$
hence
$$\big|f(x)-|f|_{\Omega(x_{0},\rho)}|\big|\leq C\rho^{\beta}.$$
Consequently,
$$\int_{\Omega(x_{0},\rho)}|f(x)-|f|_{\Omega(x_{0},\rho)}|^pdx\leq C\rho^{p\beta+n}.$$

Now we give the proof of $(3)\Longrightarrow(1)$. For any $x,y\in \Omega$, take $\Omega_{0}=\Omega(x,\rho)$ with $\rho\leq |x-y|$ and $U=\Omega(x,2|x-y|)$, define $\Omega_{k}=\Omega(x,2^{k}\rho)$ for $0\leq k\leq \tilde{k}$, where $\tilde{k}$ is the first integer such that $2^{\tilde{k}}\rho\geq |x-y|$.

Notice that for any $R_{1}=\Omega(x_{1},\rho_{1}), R_{2}=\Omega(x_{2},\rho_{2})$ with $R_{1}\subset R_{2}$ and $\rho_{2}\leq 2\rho_{1}$, we have
\begin{eqnarray*}
\begin{aligned}
\big|f_{R_{1}}-|f|_{R_{2}}\big|&=&\frac{1}{|R_{1}|}\int_{R_{1}}\big|f(z)-|f|_{R_{2}}\big|dz\leq C\rho_{1}^{\beta},
\end{aligned}
\end{eqnarray*}
and
\begin{eqnarray*}
\begin{aligned}
&\big||f|_{R_{1}}-|f|_{R_{2}}\big|=\frac{1}{|R_{1}|}\int_{R_{1}}\big||f|_{R_{1}}-|f|_{R_{2}}\big|dz\\
&\leq \frac{1}{|R_{1}|}\int_{R_{1}}\big|f(z)-|f|_{R_{1}}\big|dz+\frac{1}{|R_{1}|}\int_{R_{1}}\big|f(z)-|f|_{R_{2}}\big|dz
\leq C\rho_{1}^{\beta}
\end{aligned}
\end{eqnarray*}
Therefore,
\begin{eqnarray*}
\begin{aligned}
|f_{\Omega_{0}}-|f|_{U}|&\leq \big|f_{\Omega_{0}}-|f|_{\Omega_{1}}\big|+\sum_{k=1}^{\tilde{k}-1}\big||f|_{\Omega_{k}}-|f|_{\Omega_{k+1}}\big|+\big||f|_{\Omega_{\tilde{k}}}-|f|_{U}\big|\\
&\leq C\sum_{k=0}^{\tilde{k}}(2^{k}\rho)^{\alpha}\leq C|x-y|^{\beta}.
\end{aligned}
\end{eqnarray*}

A similar argument can be made for the point $y$ with $\Omega'_{0}=\Omega(x,\rho')$ and $V=\Omega(y,3|x-y|)$. Thus,
\begin{eqnarray*}
\begin{aligned}
&\big|f_{\Omega_{0}}-f_{\Omega'_{0}}\big|\leq \big|f_{\Omega_{0}}-|f|_{U}\big|+\big||f|_{U}-|f|_{V}\big|+\big||f|_{V}-f_{\Omega'_{0}}\big|\leq C|x-y|^{\beta}.
\end{aligned}
\end{eqnarray*}

From the differentiation theorem of Lebesgue we know that $f_{\Omega_{0}}\rightarrow f(x)$ as $\rho\rightarrow 0$, and $f_{\Omega'_{0}}\rightarrow f(y)$ in $L^{1}(\Omega)$ as $\rho'\rightarrow 0$.
It follows that
$$|f(x)-f(y)|\leq C|x-y|^{\beta}.$$
Therefore, we can conclude that $f\in C^{0,\alpha}$. Meanwhile,
\begin{eqnarray*}
\begin{aligned}
&\frac{2}{|\Omega(x,\rho)|}\int_{|\Omega(x,\rho)|}f^{-}(z)dz\leq \frac{1}{|\Omega(x,\rho)|}\int_{\Omega(x,\rho)}\big||f(z)|-f(z)\big|dz\\
&\leq \frac{1}{|\Omega(x,\rho)|}\int_{\Omega(x,\rho)}\big||f(z)|-|f|_{\Omega(x,\rho)}\big|dz+\frac{1}{|\Omega(x,\rho)|}\int_{\Omega(x,\rho)}\big|f(z)-|f|_{\Omega(x,\rho)}\big|dz\\
&\leq C\rho^{\beta}\rightarrow 0,\quad  \text{as} \quad \rho\rightarrow 0,
\end{aligned}
\end{eqnarray*}
which shows that $f^{-}(x)=0, a.e. x\in \Omega$.
\end{proof}

In fact, we can also obtain the following results. The remaining proofs are similar to the ones in Proposition \ref{Pp} and we leave the
details to the interested reader.
\begin{theorem}
Let $0<p<\infty$ and $0<\beta\leq 1$. Then $f\in C^{0,\beta}(\Omega)$ and $f\geq 0$ if and only if
$$\sup_{Q\subset Q_{0}}\inf_{c\geq 0}\rho^{-n-\beta}\int_{Q}|u(x)-c|^pdx<\infty.$$
\end{theorem}

\subsection{Characterizations of $\mathcal{\bar{L}}^{1,\lambda}(\mathbb{R}^n)$ function associated to maximal functions.}

To obtain the result of maximal function characterizations of $\mathcal{\bar{L}}^{1,\lambda}(\mathbb{R}^n)$, we need the locally boundedness of $M_{Q}$ on $\mathcal{\bar{L}}^{1,\lambda}(Q)$.

\begin{theorem}
If $f\in \mathcal{\bar{L}}^{1,\lambda}(Q)$ with $n<\lambda\leq n+1$, then so does $M_{Q}(f)$ and
\begin{equation}\label{MLIP}
\|M_{Q}f\|_{\mathcal{\bar{L}}^{1,\lambda}(Q)}\leq C\|f\|_{\mathcal{\bar{L}}^{1,\lambda}(Q)}.
\end{equation}
\end{theorem}
\begin{proof}
It follows from Theorem \ref{Cmain} that $f$ is nonnegative. Writing $F$ for the maximal function $M_{Q}f$ of $f$, we thus need to prove
\begin{equation}\label{MLIP1}
\frac{1}{|R|^{\lambda/n}}\int_{R}|F(x)-F_{R}|dx\leq C\|f\|_{\mathcal{\bar{L}}^{1,\lambda}(Q)}
\end{equation}
for every subcubes $R$ of $Q$.

Fix $R$ and let $3R$ denote the cube that is concentric with $R$ and has three times the diameter. Let $\tilde{R}$ be the smallest subcube of $Q$ containing $(3R)\cap Q$, and for each $x\in R$ let
\begin{eqnarray*}
\begin{aligned}
&F_{1}(x)=\sup\{f_{\bar{R}}:\bar{R}\subset \tilde{R} \quad \text{and}\quad  x\in \tilde{R}\},\\
&F_{2}(x)=\sup\{f_{\bar{R}}: \bar{R}\subset Q, x\in\tilde{R}\quad \text{and} \quad \bar{R}\cap(Q\backslash \tilde{R})\neq \varnothing\}.
\end{aligned}
\end{eqnarray*}
Meanwhile, if
$$D=\{x\in R: F(x)>F_{R}\}, D_{1}=\{x\in D: F_{1}(x)\geq F_{2}(x)\} \text{and} D_{2}=D\backslash D_{1},$$
then
$$\frac{1}{|R|^{\lambda/n}}\int_{R}|F(x)-F_{R}|dx=\frac{2}{|R|^{\lambda/n}}\int_{D}|F(x)-F_{R}|dx=\frac{2}{|R|^{\lambda/n}}\sum_{i=1}^{2}\int_{D_{i}}|F_i(x)-F_{R}|dx.$$
Hence we can establish the inequality \eqref{MLIP1} by
\begin{equation}\label{MLIP2}
\frac{1}{|R|^{\lambda/n}}\int_{D_{i}}|F_i(x)-F_{R}|dx\leq C\|f\|_{\mathcal{\bar{L}}^{1,\lambda}(Q)}.
\end{equation}

Consider first the case $i=1$. Since $f_{\tilde{R}}\leq F(x)$ for all $x\in R$, then $f_{\tilde{R}}\leq F_R$ so we may construct the Calder\'{o}n-Zygmund decomposition for $f$ and $\tilde{R}$ with respect to the constant $F_{R}$. If the resulting sequence of pairwise disjoint cubes is denoted by $\{R_{k}\}_{k=1}^{\infty}$ and if $\bar{R}_{k}$ denotes the "parent" cube of $R_{k}$, then the following properties hold:
\begin{itemize}
  \item [(i)] $\bigcup_{k}R_k\subset \tilde{R}$;
  \item [(ii)] $f_{\bar{R}_k}\leq F_R<f_{R_k}$;
  \item [(iii)] $|\bar{R}_k|=2^n|R_k|$;
  \item [(iv)] $f\leq F_R$ almost everywhere on $E=\tilde{R}\backslash (\bigcup_kR_k)$.
\end{itemize}

Define function $b$ and $g$ on $Q$ by
$$b=\sum_{k}(f-f_{\bar{R}_k})\chi_{R_k}, g=\sum_{k}f_{\bar{R}_k}\chi_{R_k}+f\chi_{E}.$$
So $f\chi_{\tilde{R}}=b+g$. It follows from ${ii}$ and $(iv)$ that
\begin{equation}\label{MLIP3}
\|g\|_{L^{\infty}(Q)}\leq F_R,
\end{equation}
while on the other hand the John-Nirenberg lemma and $(i)$ and $(iii)$ give
\begin{eqnarray}\label{MLIP4}
\begin{aligned}
\|b\|_{L^{2}(Q)}&=\Big\{\sum_{k}\int_{R_k}|f(x)-f_{\bar{R}_k}|^{2}dx\Big\}^{1/2}\\
&\leq \Big\{\sum_{k}|\bar{R_k}|^{\lambda/n}\frac{1}{|\bar{R}_k|^{\lambda/n}}\int_{\bar{R}_k}|f(x)-f_{\bar{R}_k}|^2dx\Big\}^{1/2}\\
&\leq C|R|^{\lambda/n}\|f\|_{\mathcal{\bar{L}}^{2,\lambda}(Q)}.
\end{aligned}
\end{eqnarray}
Now it follows from the definition of $F_1$ that
$$F_{1}\leq M_{Q}(f\chi_{\tilde{R}})=M_{Q}(b+g)\leq M_{Q}b+M_{Q}g,$$
so applying the Cauchy-Schwarz inequality we obtain
\begin{eqnarray*}
\begin{aligned}
\int_{D_{1}}F_{1}(x)dx&\leq |D_{1}|^{1/2}\|M_{Q}b\|_{L^{2}(Q)}+ |D_{1}|\|M_{Q}b\|_{L^{\infty}(Q)}\\
&\leq C|R|^{1/2}\|b\|_{L^{2}(Q)}+|D_{1}|\|g\|_{L^{\infty}(Q)}.
\end{aligned}
\end{eqnarray*}
Combining this with \eqref{MLIP3} and \eqref{MLIP4}, we obtain \eqref{MLIP2} for $i=1$.

Fix $x\in D_2$ and let $P$ be any subcube of $Q$ that contains $x$ and has nonempty intersection with $Q\backslash \tilde{R}$. Clearly $|P|\geq |R|.$ Let $P'$ be the smallest subcube of $Q$ containing both $P$ and $R$. Then $|P'|\leq 2^n|P|$. Arguing as before, we note that $f_{P'}\leq F_{R}$. Hence
$$f_{P}-F_{R}\leq f_{P}-f_{P'}\leq \frac{1}{|R|}\int_{P}|f(y)-f_{P'}|dy\leq 2^n\|f\|_{\mathcal{\bar{L}}^{2,\lambda}(Q)},$$
so taking the supremum over all such cubes $P$ we obtain
$$F_{2}(x)-F_{R}\leq C\|f\|_{\mathcal{\bar{L}}^{1,\lambda}(Q)}.$$
This establishes the case $i=2$ and the proof of the Theorem is completed.
\end{proof}

Similar to \cite{HD} and \cite{WZTmn}, it is easy to obtain that
\begin{theorem}\label{JNL}
Let $n<\lambda\leq n+1, \beta=\lambda-n$ and $\|f\|_{\bar{\mathcal{L}}^{1,\lambda}(\mathbb{R}^n)}=1$. There are constants $c_{1},c_{2}>0$, depending only $n$, such that
\begin{equation}\label{eqJNL}
\Big|\big\{x\in Q: |f(x)-(M_{Q}(f))_{Q}|>t|Q|^{\beta}\big\}\Big|\leq c_{1}e^{-c_{2}t}|Q|
\end{equation}
for all $Q\subset \mathbb{R}^n$ and all $t>0$.
\end{theorem}

Then, we conclude that
\begin{theorem}\label{JNLIP}
Let $n<\lambda\leq n+1, \beta=\lambda-n$ and $\|f\|_{\bar{\mathcal{L}}^{1,\lambda}(\mathbb{R}^n)}=1$. There are constants $c_{1},c_{2}>0$, depending only $n$, such that
\begin{equation}\label{eqJNLIP}
\Big|\big\{x\in Q: |f(x)-M_{Q}(f)(x)|>t|Q|^{\beta}\big\}\Big|\leq c_{1}e^{-c_{2}t}|Q|
\end{equation}
for all $Q\subset \mathbb{R}^n$ and all $t>0$.
\end{theorem}
\begin{proof}
By Theorem \ref{MLIP} and Proposition \ref{JNL}, we have
\begin{eqnarray*}
&&\frac{1}{|Q|}e^{c_{2}t}\Big|\big\{x\in Q: |f(x)-M_{Q}(f)(x)|>t|Q|^{\beta}\big\}\Big|\\
&&\leq \frac{1}{|Q|}e^{c_{2}t}\Big|\big\{x\in Q: |f(x)-(M_{Q}(f))_{Q}|>t|Q|^{\beta}\big\}\Big|\\
&&\qquad+\frac{1}{|Q|}e^{c_{2}t}\Big|\big\{x\in Q: |(M_{Q}(f))_{Q}-M_{Q}(f)(x)|>t|Q|^{\beta}\big\}\Big|\\
&&\leq C+C\|M_{Q}(f)\|_{\bar{\mathcal{L}}^{1,\lambda}(Q)}\\
&&\leq C+C\|f\|_{\bar{\mathcal{L}}^{1,\lambda}(Q)}\\
&&\leq C.
\end{eqnarray*}
There exist constants $c_{1}$ and $c_{2}$ such that
$$\Big|\big\{x\in Q: |f(x)-M_{Q}(f)(x)|>t|Q|^\beta\big\}\Big|\leq c_{1}e^{-c_{2}t}|Q|.$$
Thus we complete the proof of Theorem \ref{JNLIP}.
\end{proof}

Furthermore, we have
\begin{theorem}\label{JNMalpha2}
Suppose $f\in C^{0,\beta}(\mathbb{R}^n)$ with $f\geq 0$ and $0<\beta\leq 1$, then for any cube $Q$ and $t>0$, we have
$$\Big|\big\{x\in Q: |f(x)-|Q|^{-\beta/n}M_{\alpha,Q}(f)(x)|>t|Q|^{\beta}\big\}\Big|\leq c_{1}e^{-c_{2}t}|Q|.$$
where $c_{1}$ and $c_{2}$ are positive constants.
\end{theorem}

\section{Necessary conditions for commutators on Banach function spaces}\label{BANACH}

To state our results we recall some basic facts about Muckenhoupt weights and ball Banach function spaces.

We first recall the definition of $A_{p}$ weight introduced by Muckenhoupt in \cite{M1972}, which give the characterization of all weights $\omega(x)$ such
that the Hardy-Littlewood maximal operator is bounded on $L^{p}(\omega)$. For $1< p<\infty$ and a nonnegative locally integrable function $\omega$ on $\mathbb{R}^n$, $\omega$ is in the
Muckenhoupt $A_{p}$ class if it satisfies the condition
$$[\omega]_{A_{p}}:=\sup_{Q}\bigg(\frac{1}{|Q|}\int_{Q}\omega(x)dx\bigg)\bigg(\frac{1}{|Q|}\int_{Q}\omega(x)^{-\frac{1}{p-1}}dx\bigg)^{p-1}<\infty.$$
And a weight function $\omega$ belongs to the class $A_{1}$ if
$$[\omega]_{A_{1}}:=\frac{1}{|Q|}\int_{Q}\omega(x)dx\Big(\mathop\mathrm{ess~sup}_{x\in Q}\omega(x)^{-1}\Big)<\infty.$$
We write $A_{\infty}=\bigcup_{1\leq p<\infty}A_{p}$. In fact, the reverse H\"{o}lder inequality holds for $A_{p}$, that is, there exist constants $q>1$ and $C$ such that for any cube $Q$ and $\omega\in A_{p}$,
\begin{equation}\label{rhi}
\bigg(\frac{1}{|Q|}\int_{Q}\omega(x)^qdx\bigg)^{1/q}\leq \frac{1}{|Q|}\int_{Q}\omega(x)dx.
\end{equation}
For $\omega\in A_{\infty}$, there exists $0<\epsilon,L<\infty$ such that for all measurable subsets $S$ of cube $Q$,
\begin{align}\label{weight-eq1}
  \frac{\omega(S)}{\omega(Q)}\leq C\Big(\frac{|S|}{|Q|}\Big)^{\epsilon}
\end{align}
and
\begin{align}\label{weight-eq2}
\left( \frac{|S|}{|Q|}\right)^{L}\le C\frac{\omega(S)}{\omega(Q)}.
\end{align}

By a ball Banach function space $X$ whose norm $\|\cdot\|_{X}$ satisfies the following for all $f,g\in X$:

\vspace{0.2cm}
(1) $\|f\|_{X}=\||f|\|_{X}$;

\vspace{0.2cm}
(2) if $|f|\leq |g|$ a.e., then $\|f\|_{X}\leq \|g\|_{X}$;

\vspace{0.2cm}
(3) if $\{f_{n}\}\subset X$ is a sequence such that $|f_{n}|$ increases to $|f|$ a.e., then $\|f_{n}\|_{X}$ increases to $\|f\|_{X}$;

\vspace{0.2cm}
(4) if $E\subset \mathbb{R}^n$ is bounded, then $\|\chi_{E}\|_{X}<\infty$;

\vspace{0.2cm}
(5) if $E$ is bounded, then $\int_{E}|f(x)|d\mu\leq C\|f\|_{X},$ where $C=C(E,X)$.
\vspace{0.2cm}

Given a ball Banach function space $X$, there exists another ball Banach function space $X'$, called the associate space of $X$, such that for all $f\in X$,
$$\|f\|_{X}\approx \sup_{g\in X', \|g\|_{X'}\leq 1}\int_{\mathbb{R}^n}f(x)g(x)dx.$$
The associate space is equal to the dual space $X^*$ and always reflexive in many cases. Moreover, we have
\begin{equation}\label{Ref}
\int_{\mathbb{R}^n}|f(x)g(x)|dx\lesssim \|f\|_{X}\|g\|_{X'}.
\end{equation}

Let $X$ be a Banach function and define
$$\|f\|_{BMO_{X}}:=\sup_{Q}\frac{\|(b-b_{Q})\chi_{Q}\|_{X}}{\|\chi_{Q}\|_{X}}.$$
When the Hardy-Littlewood maximal operator $M$ is bounded on $X$, Ho \cite{H2012} first proved that $\|f\|_{BMO_{X}}$ is equivalent to $BMO(\mathbb{R}^n)$. Izuki and Sawano \cite{IS2017} gave another proof using the Rubio de Francia algorithm.

One can show that Lebesgue spaces, Morrey spaces, Lorentz spaces, variable Lebesgue spaces, weighted Lebesgue spaces and Orlicz
spaces are Banach function spaces. In this section, Our results relax the restriction of Banach spaces in previous to quasi-Banach spaces and extend $BMO(\mathbb{R}^n)$ to the class of $\mathcal{\bar{L}}^{1,n}(\mathbb{R}^n)$.

\begin{theorem}\label{ThmB}
Let $0<s<\infty$ and $X$ be a ball Banach function space such that the Hardy-Littlewood maximal operator $M$ is bounded on the associate space $X'$. Then $f\in \mathcal{\bar{L}}^{1,n}(\mathbb{R}^n)$ if and only if
$$\sup_{Q}\frac{\|(f-|f|_{Q})\chi_{Q}\|_{X^s}}{\|\chi_{Q}\|_{X^s}}<\infty,$$
where $X^{s}:=\{f: \|f\|_{X^s}:=\||f|^{s}\|^{1/s}_{X}<\infty\}.$
\end{theorem}
\begin{proof}
($\Rightarrow$). Let $B=\|M\|_{X'\rightarrow X'}$. Take $g\in X'$ with $\|g\|_{X'}=1$ and define a function
\begin{equation}\label{thmR}
R(g)(x):=\sum_{k=0}^{\infty}\frac{M^{k}g(x)}{(2B)^{k}}, g\in X',
\end{equation}
where $M^{k}(g)=M\circ(M^{k-1}g)$ and $M^{0}g=|g|$. Then, the function $Rg$ satisfies the following properties:

(1) $|g(x)|\leq Rg(x)$ for any  $x\in \mathbb{R}^n$;

(2) $\|Rg\|_{X'}\leq 2\|g\|_{X'}\leq 2$;

(3) $M(Rg)(x)\leq 2B Rg(x)$, that is, $Rg$ is a Muckenhoupt $A_{1}$ weight.

By the property of Muckenhoupt $A_{1}$, we know that there exist positive $q>1$ such that for any cubes $Q$,
$$\bigg(\frac{1}{|Q|}\int_{Q}Rg(x)^{q}dx\bigg)^{1/q}\leq \frac{C}{|Q|}\int_{Q}Rg(x)dx.$$
By the generalized Holder inequality \eqref{Ref}, we obtain
\begin{eqnarray*}
\begin{aligned}
\|Rg\chi_{Q}\|_{L^{q}(\mathbb{R}^n)}&=\bigg(\int_{Q}Rg(x)^{q}dx\bigg)^{1/q}\leq |Q|^{1/q-1}\int_{Q}Rg(x)dx\\
&\leq C|Q|^{1/q-1}\|Rg\|_{X'}\|\chi_{Q}\|_{X}\leq C|Q|^{1/q-1}\|\chi_{Q}\|_{X}.
\end{aligned}
\end{eqnarray*}
Thus we have
\begin{eqnarray*}
\begin{aligned}
\|(f-|f|_{Q})^s\chi_{Q}\|_{X}&\leq C\sup\bigg\{\Big|\int_{Q}(f(x)-|f|_{Q})^{s}g(x)dx\Big|:g\in X',\|g\|_{X'}\leq 1\bigg\}\\
&\leq C\sup\bigg\{\int_{Q}|f(x)-|f|_{Q}|^{s}Rg(x)dx:g\in X',\|g\|_{X'}\leq 1\bigg\}\\
&\leq C\sup\bigg\{\|(f-|f|_{Q})^s\chi_{Q}\|_{L^{q'}(\mathbb{R}^n)}\|Rg\chi_{Q}\|_{L^{q}(\mathbb{R}^n)}:g\in X',\|g\|_{X'}\leq 1\bigg\}\\
&\leq C\bigg(\frac{1}{|Q|}\int_{Q}|f-|f|_{Q}|^{sq'}dx\bigg)^{1/q'}\|\chi_{Q}\|_{X}.
\end{aligned}
\end{eqnarray*}
This yields that
$$\frac{\|(f-|f|_{Q})^s\chi_{Q}\|^{1/s}_{X}}{\|\chi_{Q}\|^{1/s}_{X}}\leq  C\bigg(\frac{1}{|Q|}\int_{Q}|f-|f|_{Q}|^{sq'}dx\bigg)^{1/(sq')}\leq C\|f\|_{\mathcal{\bar{L}}^{1,n}(\mathbb{R}^n)}.$$

$(\Leftarrow)$. For any cube $Q$,
\begin{eqnarray*}
\begin{aligned}
\int_{Q}|f(x)-|f|_{Q}|^{s}dx\leq C\|(f-|f|_{Q})^s\chi_{Q}\|_{X}\|\chi_{Q}\|_{X'}.
\end{aligned}
\end{eqnarray*}
The boundedness of $M$ on $X'$ gives us that
$$\|\chi_{Q}\|_{X}\|\chi_{Q}\|_{X'}\leq C|Q|,$$
it follows that $f\in \mathcal{\bar{L}}^{1,n}(\mathbb{R}^n)$ by Propositions \ref{p<1} and \ref{p>1}.
\end{proof}

Unfortunately, if $0<s<1$, we do not know whether or not the condition $f\in \mathcal{\bar{L}}^{1,n}(\mathbb{R}^n)$ is necessary for
\begin{eqnarray*}
\begin{aligned}
\sup_{Q}\int_{Q}|f(x)-|f|_{Q}|^{s}dx\leq C.
\end{aligned}
\end{eqnarray*}
Then, we only obtain partly results about characterizations of $\mathcal{\bar{L}}^{1,\lambda}(\mathbb{R}^n)$ function associated to maximal functions on ball Banach function space.

\begin{theorem}\label{W}
Let $1\leq s<\infty$ and $X$ be a ball Banach function space such that the Hardy-Littlewood maximal operator $M$ is bounded on the associate space $X'$. Then $f\in \mathcal{\bar{L}}^{1,n}(\mathbb{R}^n)$ if and only if
$$\sup_{Q}\frac{\|(f-M_{Q}f)\chi_{Q}\|_{X^s}}{\|\chi_{Q}\|_{X^s}}<\infty.$$
\end{theorem}
\begin{theorem}\label{D}
Let $1\leq s<\infty$ and $X$ be a ball Banach function space such that the Hardy-Littlewood maximal operator $M$ is bounded on the associate space $X'$. Then $f\in \mathcal{\bar{L}}^{1,n}(\mathbb{R}^n)$ if and only if
$$\sup_{Q}\frac{\|(f-|Q|^{-\alpha/n}M_{\alpha,Q}(f))\chi_{Q}\|_{X^s}}{\|\chi_{Q}\|_{X^s}}<\infty.$$
\end{theorem}

As applications, we obtain the necessary conditions for commutators on Banach function spaces.
\begin{theorem}
Let $X$ be a ball Banach function space such that $M$ is bounded on the associate space $X'$. If the commutator satisfy $[b,M]: X\rightarrow X$, then $b\in BMO(\mathbb{R}^n)$  with $b^{-}\in L^{\infty}(\mathbb{R}^n)$..
\end{theorem}
\begin{proof}
For any cube $Q$, we write $f=\chi_{Q}$, then for any $x\in Q$
$$M(f)(x)=1, M(bf)(x)=M_{Q}(b)(x).$$
This shows that
$[b,M](f)(x)=b(x)-M_{Q}(b)(x)$
and
$$\|[b,M](f)\|_{X}=\|(b-M_{Q}(b))\chi_{Q}\|_{X}\leq \|[b,M]\|_{X\rightarrow X}\|\chi_{Q}\|_{X}.$$
Since this is true for every cube $Q$, $b\in BMO(\mathbb{R}^n)$ by Theorem \ref{W}.
\end{proof}

\begin{theorem}\label{H}
Let $X$ and $Y$ be the ball Banach function space. Suppose that $M$ is bounded on the associate space $Y'$ and $M_{\alpha}$ is bounded from $Y'$ to $X'$. If the commutator satisfy $[b,M_{\alpha}]: X\rightarrow Y$, then $b\in BMO(\mathbb{R}^n)$ with $b^{-}\in L^{\infty}(\mathbb{R}^n)$.
\end{theorem}
\begin{proof}
For any cube $Q$, we write $f=|Q|^{-\alpha/n}\chi_{Q}$, then for any $x\in Q$
$$M(f)(x)=1, M_{\alpha}(bf)(x)=|Q|^{-\alpha/n}M_{\alpha,Q}(b)(x).$$
This shows that
$[b,M_{\alpha}](f)(x)=b(x)-|Q|^{-\alpha/n}M_{\alpha,Q}(b)(x)$
and
$$\|[b,M](f)\|_{Y}=\|(b-|Q|^{-\alpha/n}M_{\alpha,Q}(b))\chi_{Q}\|_{Y}\leq |Q|^{-\alpha/n}\|[b,M_{\alpha}]\|_{X\rightarrow Y}\|\chi_{Q}\|_{X}.$$
From the $(Y',X')$ boundedness of $M_{\alpha}$ and \cite[Lemma 2.1]{B1999}, we have
$$\|\chi_{Q}\|_{Y'}\|\chi_{Q}\|_{X}\leq C|Q|^{1-\alpha/n}.$$
We can now continue the above estimate:
$$\|(b-|Q|^{-\alpha/n}M_{\alpha,Q}(b))\chi_{Q}\|_{Y}\leq \frac{C|Q|}{\|\chi_{Q}\|_{Y'}}\leq C\|\chi_{Q}\|_{Y}.$$
Then $b\in BMO(\mathbb{R}^n)$ by Theorem \ref{D}.
\end{proof}

As we discuss above, the assumption of a geometric condition on the underlying
spaces that is closely related to the boundedness of the Hardy-Littlewood maximal
operator, and which holds in a large number of important special cases, such as the Morrey and variable Lebesgue spaces. However, Theorem \ref{H} cannot generalized to the weighted Lebesgue spaces. For example, the maximal function may not bounded on $L^{q'}(\omega^{-q'})$ for $\omega\in A_{p,q}$. In fact, for the weighted Lebesgue spaces, we have

\begin{theorem}
Let $1<p<q<\infty$, $\frac{1}{p}-\frac{1}{q}=\frac{\alpha}{n}$ and $\omega\in A_{p,q}$. Then $f\in \mathcal{\bar{L}}^{1,n}(\mathbb{R}^n)$ if and only if
$$\sup_{Q}\frac{|Q|^{\alpha/n}\|(f-|Q|^{-\alpha/n}M_{\alpha,Q}(f))\chi_{Q}\|_{L^{q}(\omega^q)}}{\|\chi_{Q}\|_{L^{p}(\omega^p)}}<\infty.$$
\end{theorem}
\begin{proof}
$(\Rightarrow)$: By Theorem \ref{JNMalpha}, we have
$$\Big|\big\{x\in Q: |f(x)-|Q|^{-\alpha/n}M_{\alpha,Q}(f)(x)|>t\big\}\Big|\leq a_{1}e^{-a_{2}t}|Q|.$$
Since $\omega\in A_{p,q}$, we have $\mu:=\omega^{q}\in A_{q}\subset A_{\infty}$. Then for any cube $Q$ and any measurable set $E$ contained in $Q$, there are positive constants $C_{0}$ and $L$ such that
$$\Big(\frac{|E|}{|Q|}\Big)^{N}\leq C\frac{\mu(E)}{\mu(Q)}.$$
This implies that
$$\mu\big(\{x\in Q:|f(x)-|Q|^{-\alpha/n}M_{\alpha,Q}(f)(x)|>\lambda\}\big)\leq Ce^{-ct}\mu(Q).$$
Hence, for any ball $Q$,
\begin{eqnarray*}
&&\|(f(x)-|Q|^{-\alpha/n}M_{\alpha,Q}(f)(x))\chi_{Q}\|^{q}_{L^{q}(\mu)}\\
&=&q\int_{0}^{\infty}\lambda^{q-1}\mu\big(\{x\in Q:|f(x)-|Q|^{-\alpha/n}M_{\alpha,Q}(f)(x)|>\lambda\}\big)d\lambda\\
&\leq&C\int_{0}^{\infty}\lambda^{q-1}e^{-ct}\mu(Q)d\lambda\\
&\leq&C\mu(Q).
\end{eqnarray*}
By H\"{o}lder inequality, we have
\begin{equation*}
|Q|\leq \bigg(\int_{Q}\omega(x)^{p}dx\bigg)^{1/p}\bigg(\int_{Q}\omega(x)^{-p'}dx\bigg)^{1/p'}.
%=\|\chi_{Q}\|_{L^{p}(\omega^{p})}\bigg(\int_{Q}\omega(x)^{-p'}dx\bigg)^{1/p'}.
\end{equation*}
Then, it follows from $\omega\in A_{p,q}$ that
\begin{eqnarray*}
\frac{\mu(Q)^{1/q}|Q|^{\alpha/n}}{\omega^{p}(Q)^{1/p}}
&\leq&|Q|^{1/p-1/q-1}\bigg(\int_{Q}\omega(x)^{q}dx\bigg)^{1/q}\bigg(\int_{Q}\omega(x)^{-p'}dx\bigg)^{1/p'}\\
&\leq&\bigg(\frac{1}{|Q|}\int_{Q}\omega(x)^{q}dx\bigg)^{1/q}\bigg(\frac{1}{|Q|}\int_{Q}\omega(x)^{-p'}dx\bigg)^{1/p'}\\
&\leq&C.
\end{eqnarray*}
Thus, $f\in \mathcal{\bar{L}}^{1,n}(\mathbb{R}^n)$ implies that
$$\frac{|Q|^{\alpha/n}}{\omega^{p}(Q)^{1/p}}\bigg(\int_{Q}|f(x)-|Q|^{-\alpha/n}M_{\alpha,Q}(f)(x)|^{q}\omega(x)^{q}dx\bigg)^{1/q}\leq C.$$

$(\Rightarrow)$: Now we prove that if there exists a constant $C$ such that for any cube $Q$,
$$\frac{1}{\omega^{p}(Q)^{1/p}}\bigg(\int_{Q}|f(x)-|Q|^{-\alpha/n}M_{\alpha,Q}(f)(x)|^{q}\omega(x)^{q}dx\bigg)^{1/q}
\leq C|Q|^{-\alpha/n}.$$
then $f\in \mathcal{\bar{L}}^{1,n}(\mathbb{R}^n)$.

When $p>1$, H\"{o}lder inequality gives us that
\begin{eqnarray*}
&&\int_{Q}|f(x)-|Q|^{-\alpha/n}M_{\alpha,Q}(f)(x)|dx\\
&\leq& \bigg(\int_{Q}|f(x)-|Q|^{-\alpha/n}M_{\alpha,Q}(f)(x)|^{p}\omega(x)^{p}dx\bigg)^{1/p}\bigg(\int_{Q}\omega(x)^{-p'}dx\bigg)^{1/p'}\\
&\leq&C|Q|^{\alpha/n}\bigg(\int_{Q}|f(x)-|Q|^{-\alpha/n}M_{\alpha,Q}(f)(x)|^{q}\omega(x)^{q}dx\bigg)^{1/q}\bigg(\int_{Q}\omega(x)^{-p'}dx\bigg)^{1/p'}\\
&\leq&C\bigg(\int_{Q}\omega(x)^{-p'}dx\bigg)^{1/p'}\bigg(\int_{Q}\omega(x)^{p}dx\bigg)^{1/p}\\
&\leq&C|Q|\bigg(\frac{1}{|Q|}\int_{Q}\omega(x)^{-p'}dx\bigg)^{1/p'}\bigg(\frac{1}{|Q|}\int_{Q}\omega(x)^{q}dx\bigg)^{1/q}\\
&\leq&C|Q|.
\end{eqnarray*}

When $p=1$, applying the definition of $A_{1,q}$, we have
\begin{eqnarray*}
&&\int_{Q}|f(x)-|Q|^{-\alpha/n}M_{\alpha,Q}(f)(x)|dx\\
&\leq& \int_{Q}|f(x)-|Q|^{-\alpha/n}M_{\alpha,Q}(f)(x)|\omega(x)dx\cdot\big\|\frac{1}{\omega}\chi_{Q}\big\|_{L^{\infty}}\\
&\leq&C\bigg(\int_{Q}|f(x)-|Q|^{-\alpha/n}M_{\alpha,Q}(f)(x)|^{q}\omega(x)^{q}dx\bigg)^{1/q}\cdot\big\|\frac{1}{\omega}\chi_{Q}\big\|_{L^{\infty}}|Q|^{\alpha/n}\\
&\leq&C|Q|.
\end{eqnarray*}
Therefore, we conclude that $f\in \mathcal{\bar{L}}^{1,n}(\mathbb{R}^n)$.
\end{proof}

Finally, we can establish the similar results for $n<\lambda\leq n+p$ and we omitted the detail.

\section{Remarks}\label{RAO}

In the bilinear setting, the linear commutator is defined by
$$[b_{1},T]_{1}(f_{1},f_{2})(x):=b_{1}T(f_{1},f_{2})(x)-T(b_{1}f_{1},f_{2})(x),$$
$$[b_{2},T]_{1}(f_{1},f_{2})(x):=b_{2}T(f_{1},f_{2})(x)-T(f_{1},b_{2}f_{2})(x)$$
and
$$[\Sigma \vec{b},T](f_{1},f_{2})(x):=[b_{1},T]_{1}(f_{1},f_{2})(x)+[b_{2},T]_{1}(f_{1},f_{2})(x).$$
The boundedness result of linear commutators of multilinear Calder\'{o}n-Zygmund operators $[\Sigma \vec{b},T]$ was shown in \cite{PT2003}. The necessity conclusion lasted a long time and the proofs in \cite{Chaf2016} treat the term $[b_{1},T]_{1}$ only. However, the boundedness of $[\Sigma \vec{b},T]$ can not implies that $[b_{i},T]_{i}$ is a bounded operator. Using some tedious calculations in applications, the linear characterization result was obtained in \cite{WZTadm}. However, it is easy to obtain the linear characterization result related to multilinear maximal operator
$$\mathcal{M}(f_{1},f_{2})(x)=\sup_{Q\ni x}\prod_{i=1}^{2}\frac{1}{|Q|}\int_{Q}|f_{i}(y_{i})|dy_{i}.$$
\begin{theorem}
Let $\vec{b}=(b_{1},b_{2}), 1<p, p_{1},p_{2}<\infty$ and $1/p=1/p_{1}+1/p_{2}$. If the linear commutator $[\Sigma \vec{b},\mathcal{M}]$ is bounded from $L^{p_{1}(\mathbb{R}^n)}\times L^{p_{2}}(\mathbb{R}^n)$ to $L^{p}(\mathbb{R}^n)$, then $b_{i}\in BMO(\mathbb{R}^n)$ with $b_{i}\in L^{\infty}(\mathbb{R}^n)$ with $i=1,2$.
\end{theorem}
\begin{proof}
From the fact that
\begin{eqnarray*}
&&\mathcal{M}(\chi_{Q},\chi_{Q})(x)=1, \\
&&\mathcal{M}(b_{1}\chi_{Q},\chi_{Q})(x)=M_{Q}(b_{1})(x),\\
&&\mathcal{M}(\chi_{Q},b_{2}\chi_{Q})(x)=M_{Q}(b_{2})(x),\qquad x\in Q.
\end{eqnarray*}
It is obvious that
$$b_{1}(x)\leq M_{Q}(b_{1})(x) \quad \text{and} \quad b_{2}(x)\leq M_{Q}(b_{2})(x).$$
Therefore,
\begin{eqnarray*}
\frac{1}{|Q|}\int_{Q}|b_{1}(x)-M_{Q}(b_{1})(x)|dx&&\leq \frac{1}{|Q|}\int_{Q}\Big|b_{1}(x)+b_{2}(x)-M_{Q}(b_{1})(x)-M_{Q}(b_{2})(x)\Big|dx\\
&&\leq \bigg(\frac{1}{|Q|}\int_{Q}\Big|b_{1}(x)+b_{2}(x)-M_{Q}(b_{1})(x)-M_{Q}(b_{2})(x)\Big|^{p}dx\bigg)^{1/p}\\
&&\leq \|[\Sigma \vec{b},\mathcal{M}]\|_{L^{p_{1}(\mathbb{R}^n)}\times L^{p_{2}}(\mathbb{R}^n)\rightarrow L^{p}(\mathbb{R}^n)}.
\end{eqnarray*}
Which shows that $b_{1}\in BMO(\mathbb{R}^n)$ with $b_{1}\in L^{\infty}(\mathbb{R}^n)$. So does $b_{2}$.
\end{proof}

In addition, for $0<p<\infty$ and $0<\lambda\leq n+p$, define the variant of the Morrey-Campanato class $\mathcal{\tilde{L}}^{p,\lambda}(\Omega)$
$$\|f\|^{p}_{\mathcal{\tilde{L}}^{p,\lambda}(\Omega)}:=\sup_{x\in \Omega,\rho>0}\rho^{-\lambda}\int_{Q}\Big|f(x)+|f|_{Q}\Big|^pdx<\infty.$$
We can see that $f\in \mathcal{\bar{L}}^{p,\lambda}(Q_0)$ if and only if $-f\in \mathcal{\tilde{L}}^{p,\lambda}(Q_0)$. Therefore

\begin{theorem}
For $1\leq p<\infty$ and $0\leq \lambda<n$, we have $L^{p,\lambda}(\Omega)\approx \mathcal{\tilde{L}}^{p,\lambda}(\Omega)$.
\end{theorem}

\begin{theorem}
Let $Q_{0}$ be a cube in $\mathbb{R}^n$. Then the following statements are equivalent:
\begin{itemize}
  \item[(i)] $f\in \mathcal{\tilde{L}}^{1,n}(Q_{0})$;
  \item[(ii)] $f\in BMO(Q_{0})$ with $f^{+}\in L^{\infty}(Q_{0})$;
  \item[(iii)] For any $Q\subset Q_{0}$ and $0<p<\infty$, there is a constant $C$ such that
  $$\frac{1}{|Q|}\int_{Q}\big||f(x)|+f_{Q}\big|^{p}dx\leq C;$$
  \item[(iv)] For every $Q\subset Q_{0}$ there is a constant $c_{Q}\geq 0$ such that
$$\frac{1}{|Q|}\int_{Q}|f(x)+c_{Q}|dx<\infty;$$
  \item[(v)] For every $Q\subset Q_{0}$, we have
$$\inf_{c\geq 0}\frac{1}{|Q|}\int_{Q}|f(x)+c|dx<\infty;$$
\end{itemize}
\end{theorem}

\begin{theorem}\label{Campanato}
Let $0<\beta\leq 1$. For the function $f\in L_{loc}(\Omega)$, the following three statements are equivalent:
\begin{itemize}
\item [(i)] $f\in C^{0,\beta}(\Omega)$ and $f\leq 0$.
\item [(ii)] There exists a constant $C_{1}$ such that
$$\big|f(x)+|f(y)|\big|\leq C_{1}|x-y|^{\beta}$$
for almost every $x$ and $y$.
\item [(iii)] There exists a constant $C_{2}$ such that for any $0<p<\infty$ and any $x_{0}\in \Omega$ and $0<\rho<diam \Omega$
$$\Bigg(\frac{1}{|\Omega(x_{0},\rho)|}\int_{\Omega(x_{0},\rho)}\big|f(x)+|f|_{\Omega(x_{0},\rho)}\big|^pdx\Bigg)^{1/p}\leq C_{2}\rho^{\beta}.$$
\item [(iv)] There exists a constant $C_{3}$ such that for any $0<p<\infty$ and any $x_{0}\in \Omega$ and $0<\rho<diam \Omega$
$$\Bigg(\frac{1}{|\Omega(x_{0},\rho)|}\int_{\Omega(x_{0},\rho)}\big||f(x)|+f_{\Omega(x_{0},\rho)}\big|^pdx\Bigg)^{1/p}\leq C_{3}\rho^{\beta}.$$
\end{itemize}
\end{theorem}

\vspace{0.3cm}
{\bf Acknowledgments} We would like to thank the anonymous referee for his/her comments.

\vspace{0.5cm}

\noindent Dinghuai Wang, Lisheng Shu

\smallskip

\noindent  School of Mathematics and Statistics, Anhui Normal University, Wuhu, 241002, China

\smallskip

\noindent {\it E-mails}: \texttt{Wangdh1990@126.com; shulsh@mail.ahnu.edu.cn.}

\end{document}